\documentclass[11pt,reqno]{amsart}
\usepackage{amsmath}
\usepackage{amssymb}
\usepackage{amsthm}
\usepackage{enumerate}
\usepackage[usenames]{color}
\usepackage{eepic,epic}
\usepackage{url} 

\usepackage{comment} 
\usepackage{makecell}
\usepackage{epstopdf}

\usepackage{graphicx}

\newtheorem{thm}{Theorem}[section]

\newtheorem{lem}[thm]{Lemma}

\theoremstyle{definition}
\newtheorem{defn}[thm]{Definition}

\theoremstyle{remark}

\numberwithin{equation}{section}

\newtheorem{ass}{Assumption}

\newcommand{\ep}{\epsilon}

\begin{document}
\title[]
{A tight bound on the stepsize of the Decentralized gradient descent 
} 

 \author{Woocheol Choi}
\address{Department of Mathematics, Sungkyunkwan University, Suwon 440-746, Republic of Korea}
\email{choiwc@skku.edu}

\subjclass[2010]{Primary 65K10, 90C26 }

\keywords{Distributed optimization, Gradient descent, Sharp range}

\maketitle

\begin{abstract} In this paper, we consider the decentralized gradinet descent (DGD) given by
\begin{equation*}
x_i (t+1) = \sum_{j=1}^m w_{ij} x_j (t) - \alpha (t) \nabla f_i (x_i (t)).
\end{equation*}
We find a sharp range of the stepsize $\alpha (t)>0$ such that the sequence $\{x_i (t)\}$ is uniformly bounded when the aggregate cost $f$ is assumed be strongly convex with smooth local costs which might be non-convex. Precisely, we find a tight bound $\alpha_0 >0$ such that the states of the DGD algorithm is unfiromly bounded for non-increasing sequence $\alpha (t)$ satisfying $\alpha (0) \leq \alpha_0$. The theoretical results are also verified by numerical experiments.
\end{abstract}

\section{Introduction}

In this work, we consider the distributed optimization 
 \begin{equation}\label{prob}
 \min_{x}~f(x):= \frac{1}{m}\sum_{k=1}^{m} f_k (x),
 \end{equation}
 where $m$ denotes the number of agents and $f_k : \mathbb{R}^n \rightarrow \mathbb{R}$ is a differentiable local cost only known to agent $k$ for each $1 \leq k \leq m$.  The decetralized gradient descent is given as
 \begin{equation}\label{eq-1-1}
 x_{k}(t+1) = \sum_{j=1}^m w_{kj} x_j (t) - \alpha (t) \nabla f_k (x_k (t)).
 \end{equation}
 Here $\alpha (t) >0$ is a stepsize and $x_k (t)$ denotes the variable of agent $k$ at time instant $t \geq  0$. 
 The communication pattern among agents in \eqref{prob} is determined by an undirected graph $\mathcal{G}=(\mathcal{V},\mathcal{E})$, where each node in $\mathcal{V}$ represents each agent, and each edge $\{i,j\} \in \mathcal{E}$ means $i$ can send messages to $j$ and vice versa. The value $w_{ij}$ is a nonnegative weight value such that $w_{ij} >0$ if and only if $\{i,j\} \in \mathcal{E}$ and the matrix $W =\{w_{ij}\}_{1 \leq i,j\leq m}$ is doubly stochastic. 

This algorithm has recieved a lot of attentions from researchers in various fields. In particular, the algorithm has been a pivotal role in the development of several methods, containing online distributed gradient descent method \cite{YXL, CB}, the stochastic decentralized gradient descent \cite{POP}, and multi-agent Reinforcement Learning \cite{DMR, ZADR}. It was also extended to nested communication-local computation algorithms \cite{BBKW, CKY}.
 
For the fast convergence of the algorithm \eqref{eq-1-1}, it is important to choose a suitable sequence of the stepsize $\alpha(t)$.  It is often advantageous to choose a possibly large stepsize as the convergence may become faster as the stepsize gets larger in a stable regime.

 When it comes to the cases $m=1$, the algorithm is reduced to the gradient descent algorithm given as
\begin{equation}\label{eq-1-0}
x(t+1) = x(t) - \alpha (t) \nabla f(x(t)),
\end{equation}
and we recall a well-known convergence result in the following theorem.
\begin{thm}\label{thm-1-0}
Assume that $f$ is $\mu$-strongly convex and $L$-smooth. Suppose that the stepsize of \eqref{eq-1-0} satisfies $\alpha (t) \leq \frac{2}{\mu +L}$. Then the sequence $\{x(t)\}_{t \geq 0}$ is bounded. Moreover,
\begin{equation*}
\|x(t+1) - x_*\|^2 \leq  \Big( 1- \frac{2L\mu \alpha(t)}{L+\mu}\Big)  \|x(t) -x_*\|^2.
\end{equation*}
 \end{thm}
 
 Although the convergence property of the algorithm \eqref{eq-1-1} has been studied extensively (\cite{NO, NO2, RNV, YLY, CK}), the sharp range of $\alpha (t)$ for the convergence of the algorithm \eqref{eq-1-1} has not been completely understood, even when each cost $f_k$ is a  quadratic form.

In the early stage, the convergene of the algorithm \eqref{eq-1-1} was studied with assuming that $\|\nabla f_i \|_{\infty} <\infty$ and each function $f_i$ is convex for $1 \leq i \leq m$. Nedi{\'c}-Ozdaglar \cite{NO} showed that for the algorithm \eqref{eq-1-1} with the stepsize $\alpha(t) \equiv \alpha$, the cost value $f(\cdot)$ at an average of the iterations converges to an $O(\alpha)$-neighborhood of an optimal value of $f$.  Nedi{\'c}-Ozdaglar \cite{NO2} proved that the algorithm \eqref{eq-1-1} converges to an optimal point if the stepsize satisfies $\sum_{t=1}^{\infty}\alpha (t) = \infty$ and $\sum_{t=1}^{\infty}\alpha (t)^2 <\infty$. 
In the work of Chen \cite{I}, the algorithm \eqref{eq-1-1} with stepsize $\alpha (t) = c/t^p$ with $0<p<1$ was considered and the convergene rate was achieved as $O(1/t^p)$ for $0<p<1/2$, $O(\log t/\sqrt{t})$ for $p=1/2$, and $O(1/t^{1-p})$ for $1/2<p<1$. 

Recently, Yuan-Ling-Yin \cite{YLY} established the convergence property of the algorithm \eqref{eq-1-1} without the gradient bound assumption. Assuming that each local cost function is convex and the total cost is strongly convex, they showed that the algorithm \eqref{eq-1-1} with constant stepsize $\alpha (t) \equiv \alpha$ converges exponentially to an $O(\alpha)$-neighborhood of an optimizer $x_*$ of \eqref{prob}.  Recently, the work \cite{CK} obtained the convergence property of the algorithm \eqref{eq-1-1} for a general class of non-increasing stepsize $\{\alpha (t)\}_{t \in \mathbb{N}_0}$ given as $\alpha (t) = a/(t+w)^p$ for $a>0$, $w \geq 1$ and $0< p \leq 1$ assuming the strong convexity on the total cost function $f$, with cost functions $f_i$ not necessarily being convex.

To discuss the convergence property of \eqref{eq-1-1}, it is convenient to state the following definition.
\begin{defn} The sequence $\{x_k (t)\}$ of \eqref{eq-1-1} is said to be uniformly bounded if there exists a value $R>0$ such that
\begin{equation*}
\|x_i (t)\| \leq R 
\end{equation*}
for all $t \geq 0$ and $1 \leq i \leq m$.
 \end{defn}
We state the following convergence results  established in the works \cite{YLY, CK}.
\begin{thm}[\cite{YLY, CK}]
Assume that  each cost $f_k$ is $L$-smooth and the aggregate cost $f$ is $\mu$-strongly convex. Suppose also that the sequence $\{x_k (t)\}_{t \geq 0}$ of \eqref{eq-1-1} is uniformly bounded   and $\alpha (0) \leq \frac{2}{\mu +L}$. Then we have the following results:
\begin{enumerate}
\item If $\alpha (t) \equiv \alpha$, then the sequence $x_k (t)$ converges exponentially to an $O\Big(\frac{\alpha}{1-\beta}\Big)$ neighborhood of $x_*$. 
\item If $\alpha (t) = \frac{a}{(t+w)^{p}}$ for some $a >0, w\geq 0$ and $p \in (0,1]$, then  the sequence $x_k (t)$ converges to $x_*$ with the following rate
\begin{equation*}
\|x_i (t) -x_*\| = O \Big( \frac{1}{t^p}\Big).
\end{equation*}
\end{enumerate}
\end{thm}
We remark that a uniform bound assumption for the sequence $\{x_k (t)\}_{t \geq 0}$ is required in the above result, which is contrast to the result of Theorem \ref{thm-1-0}. In fact, the boundedness property of \eqref{eq-1-1} has been obtained under an additional restriction on the stepsize as in the following results.
\begin{thm}[\cite{YLY}]\label{thm-1-2}
  Assume that $f_j$ is convex and $L$-smooth. Suppose that the stepsize is constant $\alpha (t) =\alpha$ and $\alpha \leq \frac{(1+\lambda_m (W))}{L}$, then the sequence $\{x(t)\}$ is uniformly bounded. Here $\lambda_m (W)$ denotes the smallest eigenvalue of the matrix $W$.
\end{thm}

\begin{thm}[\cite{CK}]\label{thm-1-10} Assume that $f_j$ is $L$-smooth and $f$ is $\mu$-strongly convex. Let $\eta = \frac{\mu L}{\mu +L}$ and suppose that $\alpha (t) \leq \frac{\eta (1-\beta)}{L (\eta +L)}$. Then the sequence $\{x_i (t)\}$ is uniformly bounded.
\end{thm}
 In the above results, we note the following inequality
\begin{equation*}
 \frac{\eta (1-\beta)}{L (\eta +L)}< \frac{1+\lambda_m (W)}{L}.
\end{equation*}
Therefore, the result of Theorem \ref{thm-1-2} establishes the boundedness property for a broader range of the constant stepsize with assuming the convexity for each local cost. Meanwhile, the result of Theorem \ref{thm-1-10} establishes the boundedness property for time varying stepsize without assuming the convexity on each local cost, but the range of the stepsize is more restrictive. Having said this, it is natural to consider the following questions:

\medskip 

\noindent  \textbf{Question 1:} \emph{Does the result of Theorem \ref{thm-1-2} hold for time-varying stepsize? and can we moderate the convexity assumption on each cost?}

\medskip

\noindent \textbf{Question 2:} \emph{Can we extend the range of $\alpha (t)$ in the result of Theorem \ref{thm-1-10} with an additional information of the cost functions?}

\medskip 

We may figure out the importance of these questions by considering an example: Consider $m=3$ and the functions $f_1, f_2, f_3 : \mathbb{R}^2 \rightarrow \mathbb{R}$ defined as
\begin{equation}\label{eq-1-62}
f_1 (x,y) =f_2 (x,y)= \frac{L}{2} x^2 + \frac{\mu}{2} y^2\quad \textrm{and}\quad f_3 (x,y) = -\frac{\ep}{2} x^2 + \frac{\mu}{2} y^2,
\end{equation}
where $L \gg \mu_0 >0$ and $\ep\geq 0$ is small enough. Then $f_k$ is $L$-smooth for $1 \leq k \leq 3$ and the aggregate cost $f$ is $\mu$-strongly convex. If we apply the above results, we have the following results:

\noindent If $\ep =0$, then the boundedness property holds by Theorem \ref{thm-1-2} for constant stepsize $\alpha (t) \equiv \alpha$ satisfying 
\begin{equation*}
\alpha \leq \frac{1 + \lambda_m (W)}{L}.
\end{equation*}
If $\ep >0$, then the boundedness property holds by Theorem \ref{thm-1-10}   for the stepsize $\alpha(t)$ satisfying
\begin{equation*}
\alpha (t)  \leq \frac{\eta (1-\beta)}{L (\eta +L)}.
\end{equation*}
Here we see that the range guaranteed by the above theorems drastically changes as the value $\ep$ becomes positive from zero.

  
  \medskip


\medskip 

\noindent The purpose of this paper is to provide reasonable answers to the above-mentioned questions. For this we consider the function $G_{\alpha}:\mathbb{R}^{nm} \rightarrow \mathbb{R}$ for $\alpha >0$ defined by
\begin{equation*}
\begin{split}
G_{\alpha}(\mathbf{x})& = \frac{\alpha}{m} \sum_{k=1}^m f_k (x_k) +\frac{1}{2} \mathbf{x}^T (I-W\otimes 1_n ) \mathbf{x}
\\
&=: \alpha\mathbf{F}(\mathbf{x})+ \frac{1}{2} \mathbf{x}^T (I-W\otimes 1_n ) \mathbf{x}.
\end{split}
\end{equation*}
As well-known, the algorithm \eqref{eq-1-1} is then written as 
\begin{equation}\label{eq-1-4}
\mathbf{x}(t+1) = \mathbf{x}(t) - \nabla G_{\alpha (t)} (\mathbf{x}(t)).
\end{equation}
For each $\alpha>0$ we define the following function class
\begin{equation}\label{eq-1-20}
\mathbf{G}_{\alpha} = \Big\{ (f_1, \cdots, f_m) \in (C_1 (\mathbb{R}^n))^m~\Big|~ \textrm{the function $G_{\alpha}$ is strongly convex }\Big\}.
\end{equation}
The following is the main result of this paper.
\begin{thm}\label{thm-3-3}Suppose that $F= (f_1, \cdots, f_m) \in A_{\alpha_0}$  for some $\alpha_0 >0$ and $f_k$ is $L$-smooth for each $1 \leq k \leq m$. Assume that the sequence of stepsize $\{\alpha (t)\}_{t \in \mathbb{N}}$ is a non-increasing sequence satisfying
\begin{equation}\label{eq-1-61}
\alpha (0) \leq \min\Big\{ \frac{1+\lambda_m (W)}{L},~\alpha_0\Big\}.
\end{equation}
Then the sequence $\{\mathbf{x}(t)\}$ is uniformly bounded.
\end{thm}
This result naturally extends the previous works \cite{YLY, CK}. We mention that the result \cite{YLY} was obtained for constant stepsize $\alpha (t) \equiv \alpha \leq \frac{1+\lambda_m (W)}{L}$ under the assumption that each cost $f_k$ is convex. Meanwhile, the result \cite{CK} was obtained only assuming that the aggregate cost $f$ is strongly convex but the stepsize range is more conservative. We summarize the results of \cite{YLY, CK} and this work in Table \ref{known results1}. 

\begin{table}[ht]
\centering
\begin{tabular}{|c|c|c|c| }\cline{1-4}
& Convexity condition& Type of stepsize & Bound on stepsize  
\\
\hline
&&&\\[-1em]
\cite{YLY} & Each $f_j$ is convex  & \makecell{constant \\ $\alpha (t) \equiv \alpha$} & $\alpha \leq \frac{(1+\lambda_m (W))}{L}$ 
\\
&&&\\[-1em]
\hline
&&&\\[-1em]
\makecell{\cite{CK}} & $f$ is $\mu$-strongly convex & any sequence &$\alpha (t) \leq \frac{\eta (1-\beta)}{L (\eta +L)}$
\\
&&&\\[-1em]
\hline
&&&\\[-1em]
\makecell{This\\work} & $G_{\alpha_0}$ is strongly convex & \makecell{non-increasing\\ sequence} &$\alpha (t) \leq \min\Big\{ \frac{1+\lambda_m (W)}{L},~\alpha_0\Big\}.$
\\
\hline
\end{tabular}
\vspace{0.1cm}
\caption{This table summarizes the conditions for the uniform boundedness of the sequence $\mathbf{x}(t)$ of DGD algorithm \eqref{eq-1-4}.}\label{known results1}
\end{table} 

For the example \eqref{eq-1-62}, we set $L=10$, $\mu =1$, and
\begin{equation*}
W = \begin{pmatrix} 0.4 & 0.3 & 0.3 \\ 0.3 &0.3 & 0.4 \\ 0.3 & 0.4 & 0.3\end{pmatrix}.
\end{equation*}
Then for $0\leq \ep \leq 10$ we have 
\begin{equation*}
\alpha_S:=\frac{\eta (1-\beta)}{L (\eta +L)} = 0.0075, \quad \frac{1 + \lambda_m (W)}{L} = 0.2.
\end{equation*}
Finally the value $\alpha_0$ of Theorem \ref{thm-3-3} is computed for each $\ep \in (0,10]$. The graph of Figure \ref{fig2} shows that the value $\alpha_0$ is larger than $(1+\lambda_m (W))/L$ for small $\ep>0$, but it becomes smllaer than the latter value $\ep>0$ is larger than $3$. For the detail computing these values, we refer to Section \ref{sec-5}.

\begin{figure}[h]
	\centering
		\includegraphics[width=7cm]{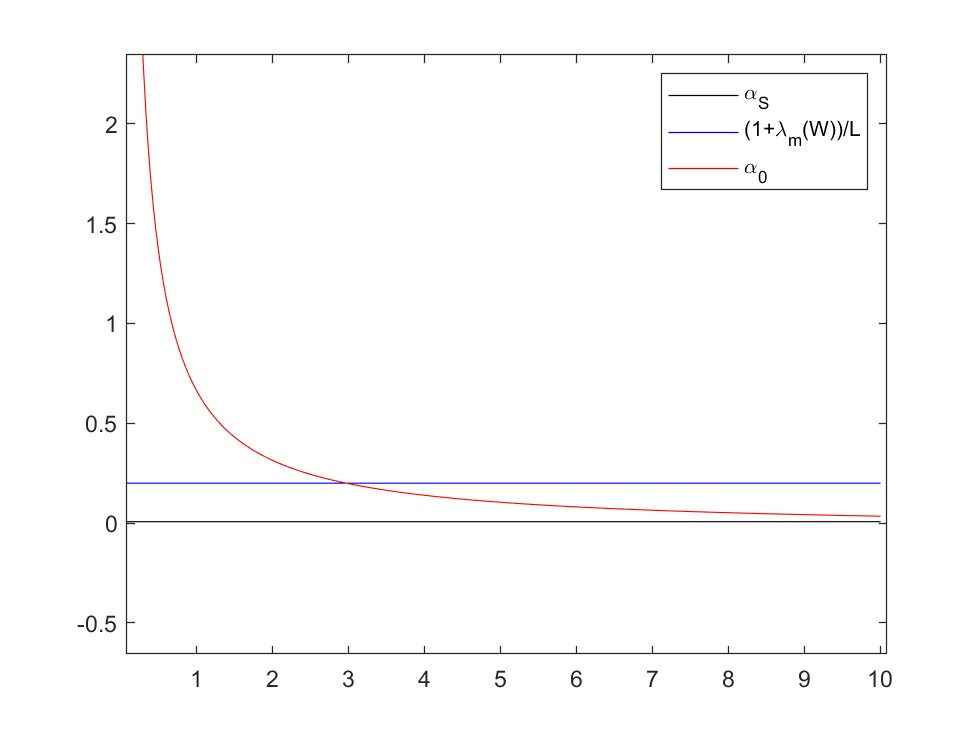} 	\caption{The figure corresponds to the graph of values $\alpha_0$, $\alpha_S$ and $\frac{1+\lambda_m (W)}{L}$ with respect to $\ep \geq 0$.}
 \label{fig2}
\end{figure}

The function class $\mathbf{G}_{\alpha}$ is closely related to the function class that the aggregate cost $f$ is strongly convex. To explain the relation, for each $\alpha>0$ we let
\begin{equation*}
S_{\alpha}  = \{ (f_1, \cdots, f_m) \in (C_1 (\mathbb{R}^n)^m~\mid~ \textrm{the function $f$ is $\alpha$-strongly convex}\}.
\end{equation*}
Then, for given $F = (f_1, \cdots, f_m) \in (C_1 (\mathbb{R}^n)^m$, we will prove that the following relation between the class $\mathbf{G}_{\alpha}$ and $S_{\alpha}$ in Section \ref{sec-2}:
\begin{enumerate}
\item If $f_k$ is quadratic and convex for each $1 \leq k \leq m$, then 
\begin{equation*}
F \in S_{\mu} \textrm{~for some~} \mu >0 \Longrightarrow  F \in \mathbf{G}_{\alpha}  \textrm{~for all~} \alpha >0.
\end{equation*}
\item If $f_k$ is quadratic for each $1 \leq k \leq m$, then
\begin{equation*}
F \in S_{\mu} \textrm{~for some~} \mu >0 \Longrightarrow  F \in \mathbf{G}_{\alpha}  \textrm{~for some~} \alpha >0.
\end{equation*}
\item It always holds that
\begin{equation*}
F \in \mathbf{G}_{\alpha}\textrm{~for some~} \alpha >0~\Longrightarrow ~F \in S_{\mu}~\textrm{ for some} ~\mu >0.
\end{equation*}
\end{enumerate} 

For given $F \in \mathbf{G}_{\alpha}$, let us denote the optimal point of $G_{\alpha}$ by $\mathbf{x}_*^{\alpha}$. In order to prove the above result, we exploit the fact that the algorithm \eqref{eq-1-1} at step $t$ is interpreted as the gradient descent of $G_{{\alpha}(t)}$ as in \eqref{eq-1-4}. Using this fact, it is not difficult to derive the result of \mbox{Theorem \ref{thm-3-3}} if the stepsize is given by a constant $\alpha (t) \equiv \alpha >0$ since the gradient descent descent \eqref{eq-1-4} converges to the minimizer $\mathbf{x}_*^{\alpha}$. However, when the stepsize is varying, then we may not interprete \eqref{eq-1-4} as a gradient descent algorithm of a single objective function. In order to handle this case, we prove the continuity and boundedness property of $x_{\alpha}$ with respect to $\alpha$. Precisely, we obtain the following result.
\begin{thm}\label{lem-1-2} Assume that $G_{\alpha_0}$ is $\mu$-strongly convex for some $\alpha_0 >0$ and $\mu >0$. Then the following results hold:
\begin{enumerate}
\item For $\alpha \in (0,\alpha_0]$ we have
\begin{equation*}
\|\mathbf{x}_*^{\alpha}\|^2 \leq \frac{2\alpha_0}{\mu} f (0).
\end{equation*} 
\item  For all $\alpha, \beta \in (0,\alpha_0]$ we have
\begin{equation*}
\|\mathbf{x}_*^{\alpha} - \mathbf{x}_*^{\beta}\| \leq  \frac{2\alpha_0 C_1|\beta -\alpha|}{\mu\beta},
\end{equation*}
where  
\begin{equation*}
C_1 = \sup_{s\in [0,1]} \sup_{\alpha, \beta \in (0, \alpha_0]}\Big\|\nabla \mathbf{F} ( \mathbf{x}_*^{\beta}  + s(\mathbf{x}_*^{\alpha}- \mathbf{x}_*^{\beta}))\Big\|.
\end{equation*}
\end{enumerate}
\end{thm}
In the above result, we remark that the constant $C_1$ is bounded since $\mathbf{F}$ is smooth and $\mathbf{x}_*^{\alpha}$ are uniformly bounded for $\alpha \in (0, \alpha_0]$. We will make use of this result to prove Theorem \ref{thm-3-3}.    
 
 The rest of this paper is organized as follows. Section \ref{sec-2} is devoted to study the function class $\mathbf{G}_{\alpha}$. In Section \ref{sec-3} we exploit the boundedness and continuity property of the optimizer of $G_{\alpha}$ with respect to $\alpha$. Based on the property, we prove the main result in Section \ref{sec-4}. Section 5 provides some numerical experiments supporting the result of this paper.

\section{Properties of the class $\mathbf{G}_{\alpha}$}\label{sec-2}

In this section,  we study the function classe $\mathbf{G}_{\alpha}$ defined in \eqref{eq-1-20}. Before this, we introduce two standard assumptions on the graph $\mathcal{G}$ ans its associated weight $W = \{w_{ij}\}_{1\leq i,j \leq m}$.

\begin{ass}\label{graph}
The communication graph $\mathcal{G}$ is undirected and connected, i.e., there exists a path between any two agents.
\end{ass}
We define the mixing matrix $W = \{w_{ij}\}_{1 \leq i,j \leq m}$ as follows. The nonnegative weight $w_{ij}$ is given for each communication link $\{i,j\}\in \mathcal{E},$ where $w_{ij}\neq0$ if $\{i,j\}\in\mathcal{E}$ and $w_{ij} = 0$ if $\{i,j\}\notin\mathcal{E}$. In this paper, we make the following assumption on the mixing matrix $W$. 
\begin{ass}\label{ass-1-1}
The mixing matrix $W = \{w_{ij}\}_{1 \leq i,j \leq m}$ is doubly stochastic, i.e., $W\mathbf{1}=\mathbf{1}$ and $\mathbf{1}^T W = \mathbf{1}^T$. In addition, $w_{ii}>0$ for all $i \in \mathcal{V}$. 
\end{ass}
 
In the following theorem, we study the function class $\mathbf{G}_{\alpha}$ when each local cost is a quadratic form.

\begin{thm}Assume that each cost is a quadratic form $f_k (x) =\frac{1}{2}x_k^T A_k x_k$.  
\begin{enumerate}
\item Assume that each $f_k$ is convex and $f$ is $\mu$-strongly convex. Then $G_{\alpha}$ is strongly convex for any $\alpha >0$. 
\item Assume that $f$ is $\mu$-strongly convex. Then there exists a value $\alpha_0>0$ such that for $\alpha >\alpha_0$, $G_{\alpha}$ is strongly convex.
\end{enumerate}
\end{thm}
\begin{proof}
Then $G_{\alpha}$ is given as
\begin{equation*}
G_{\alpha}(\mathbf{x}) = \frac{\alpha}{2m} \sum_{k=1}^m x_k^T A_k x_k + \frac{1}{2} \mathbf{x}^T (I-W \otimes 1_n) \mathbf{x}.
\end{equation*}
 Choose a constant $Q>0$  and $\bar{Q}>0$ such that
\begin{equation}\label{eq-2-12}
\Big\| \sum_{k=1}^n A_k u_k \Big\| \leq Q \|u\| \quad \forall~ u \in \mathbb{R}^{mn}
\end{equation}
and
\begin{equation}\label{eq-2-13}
\Big\| (A_1 u_1, \cdots, A_m u_m)\Big\| \leq \bar{Q}\|u\|\quad \forall~ u \in \mathbb{R}^{nm}.
\end{equation}
Let $\mathbf{x} = 1\otimes \bar{x} + u$ with $u = \mathbf{x}- 1 \otimes \bar{x}$ and $
\bar{x} = \frac{1}{n}\sum_{k=1}^n x_k$. Using that $W 1_m =1_m$, we find 
\begin{equation}\label{eq-2-60}
G_{\alpha}(\mathbf{x}) = \frac{\alpha}{2m} \sum_{k=1}^n (\bar{x}+u_k)^T A_k (\bar{x}+u_k) + \frac{1}{2} u^T (I-W\otimes 1_n) u.
\end{equation}
  Take a constant $c>0$ and consider
  \begin{equation*}
  Y_c =\{ \mathbf{x}= 1 \bar{x} + u ~:~ \|u\| \leq c\|\bar{x}\|\}\quad \textrm{and}\quad    Z_c =\{ \mathbf{x}= 1 \bar{x} + u ~:~ \|u\| \geq c\|\bar{x}\|\}.
  \end{equation*}
If $\mathbf{x} \in Y_c$, then we have
\begin{equation}\label{eq-2-10}
\|\mathbf{x}\|^2 = n\|\bar{x}\|^2 + \|u\|^2 \leq (n+c^2) \|\bar{x}\|^2.
\end{equation}
For $\mathbf{x} \in Z_c$, it holds that 
\begin{equation}\label{eq-2-11}
\|\mathbf{x}\|^2 = n\|\bar{x}\|^2 + \|u\|^2 \leq \frac{(n+c^2)}{c^2} \|u\|^2.
\end{equation}
Now we proceed to prove (1). For $\mathbf{x} \in Y_c$ we use \eqref{eq-2-12} to estimate 
\eqref{eq-2-60} as
\begin{equation*}
\begin{split}
G_{\alpha}(\mathbf{x})& = \frac{\alpha}{2m}\Big[ \bar{x}^T \Big( \sum_{k=1}^m A_k\Big) \bar{x} + 2 \bar{x} \Big( \sum_{k=1}^m A_k u_k\Big) + \sum_{k=1}^m u_k^T A_k u_k \Big] + \frac{1}{2} u^T (I- W\otimes I_n) u
\\
&\geq \frac{\alpha}{2}\Big[ \mu \|\bar{x}\|^2 - 2Q \|\bar{x}\| \|u\|\Big]
\\
& \geq \frac{\alpha}{2} [\mu - 2Qc] \|\bar{x}\|^2.
\end{split}
\end{equation*}
Using \eqref{eq-2-10} here, we find
\begin{equation*}
G_{\alpha}(\mathbf{x}) \geq \frac{\alpha (\mu -2Qc)}{2 (n+c^2)} \|\mathbf{x}\|^2.
\end{equation*}
For $\mathbf{x} \in Z_c$ we have
\begin{equation*}
G_{\alpha}(\mathbf{x}) \geq (1-\beta) \|u\|^2 \geq \frac{(1-\beta)c^2}{n+c^2} \|\mathbf{x}\|^2.
\end{equation*}
Here we used the fact that
\begin{equation*}
y^T (I-W) y \geq (1-\beta) \|y\|^2 
\end{equation*}
for any $y \in \mathbb{R}^m$ satisfying $\langle y, 1_m\rangle =0$, where $\beta \in (0,1)$ is the second largest \mbox{eigenvalue of $W$.}

Combining the above two estimates, we find
\begin{equation*}
G_{\alpha}(\mathbf{x}) \geq  \min\Big\{\frac{(1-\beta)c^2}{n+c^2}, \frac{\alpha (\mu -2Qc)}{2 (n+c^2)}\Big\} \|\mathbf{x}\|^2,
\end{equation*}
which provesthe first assertion (1). 

Next we prove (2). For $\mathbf{x} \in Y_c$ we have
\begin{equation*}
\begin{split}
G_{\alpha}(\mathbf{x})& \geq \frac{\alpha}{2} \Big[ \mu \|\bar{x}\|^2 - 2 \|\bar{x}\| \Big\| \sum_{k=1}^m A_k u_k \Big\| - \|u\| \Big\|\sum_{k=1}^n A_k u_k\Big\|\Big]
\\
& \geq \frac{\alpha}{2} \Big[ \mu \|\bar{x}\|^2 - 2Q \|\bar{x}\|\|u\| -\bar{Q}\|u\|^2\Big]
\\
& \geq\frac{\alpha}{2} \Big[ \mu \|\bar{x}\|^2 - (2cQ + c^2 \bar{Q}) \|\bar{x}\|^2\Big].
\end{split}
\end{equation*}
For $\mathbf{x} \in Z_c$ we estimate
\begin{equation*}
\begin{split}
G_{\alpha}(\mathbf{x}) & \geq \frac{\alpha}{2} \sum_{k=1}^m \bar{x}^T A_k \bar{x} + \alpha \sum_{k=1}^m \bar{x} (A_k u_k) + \frac{\alpha}{2} \sum_{k=1}^m u_k^T A_k u_k + (1-\beta) \|u\|^2
\\
& \geq \frac{\alpha \mu}{2}\|\bar{x}\|^2 - \alpha Q\|\bar{x}\| \|u\| - \frac{\alpha \bar{Q}}{2} \|u\|^2 + (1-\beta) \|u\|^2
\\
& \geq \frac{\alpha \mu}{2}\|\bar{x}\|^2 - \Big( \frac{\alpha Q}{c} + \frac{\alpha \bar{Q}}{2}\Big) \|u\|^2 + (1-\beta) \|u\|^2.
\end{split}
\end{equation*}
This gives the following estimate
\begin{equation*}
G_{\alpha}(\mathbf{x})\geq \frac{c^2}{n+c^2}\Big[(1-\beta) -\Big( \frac{\alpha Q}{c} + \frac{\alpha \bar{Q}}{2}\Big) \Big] \|\mathbf{x}\|^2.
\end{equation*}
This completes the proof of the second assertion (2).
\end{proof}

We also have the following result.
\begin{thm}\label{thm-2-2} If $G_{\alpha}$ is $\mu$-stronlgy convex for some $\alpha >0$, then $f$ is $\frac{\mu}{\alpha}$-stronlgy convex.
\end{thm}
\begin{proof}
Let $\mathbf{x}= (x, \cdots, x)$ and $\mathbf{y} = (y, \cdots, y)$. The strongly convexity of $G_{\alpha}$ yields that
\begin{equation}\label{eq-1-5}
G_{\alpha}(\mathbf{y}) \geq G_{\alpha}(\mathbf{x}) + (\mathbf{y}-\mathbf{x}) \nabla G_{\alpha}(x) + \frac{\mu}{2}\|\mathbf{y}-\mathbf{x}\|^2.
\end{equation}
We have
\begin{equation*}
G_{\alpha}(\mathbf{x}) = \frac{\alpha}{m} \sum_{k=1}^m f_k (x), \quad G_{\alpha}(\mathbf{y}) = \frac{\alpha}{m} \sum_{k=1}^{m} f_k (y).
\end{equation*}
Also, 
\begin{equation*}
\begin{split}
\nabla G_{\alpha}(\mathbf{x}) &= \frac{\alpha}{m} (\nabla f_1 (x), \cdots, \nabla f_m (x)) + (I-W) \mathbf{x}
\\
& = \frac{\alpha}{m}   (\nabla f_1 (x), \cdots, \nabla f_m (x)).
\end{split}
\end{equation*}
Using these equalities in \eqref{eq-1-5} we find
\begin{equation*}
f(y) \geq f(x) + (y-x) \nabla f(x) + \frac{\mu}{2\alpha}\|y-x\|^2,
\end{equation*}
and so the function $f$ is $\frac{\mu}{2\alpha}$-strongly convex. The proof is done.
\end{proof}

\

 \section{Uniform bound and smoothness property of $x_*^{\alpha}$ with respect to $\alpha >0$}\label{sec-3}
 


In this section, we exploit the property of the minimizer $\mathbf{x}_{\alpha}$ of the function $G_{\alpha}$. Under the strongly convexity assumption on $G_{\alpha_0}$ for some $\alpha_0$, we will show that the optimizers $\mathbf{x}_*^{\alpha} \in \mathbb{R}^n$ is uniformly bounded for $\alpha \in (0, \alpha_0]$ and also locally Lipschitz continuous with respect to $\alpha$.

\begin{lem}\label{lem-2-1}
We assume that $G_{\alpha}$ is $\mu$-strongly convex for some $\alpha >0$. Then, the function $G_{\beta}$ is $\frac{\beta\mu}{\alpha}$-strongly convex for any $\beta \in (0,\alpha]$.
\end{lem}
\begin{proof}For $\beta \in (0,\alpha]$, we express the function $G_{\beta}$ as
\begin{equation*}
\begin{split}
G_{\beta}(\mathbf{x})& = \beta \sum_{k=1}^n f_k (x_k) + \frac{1}{2} \mathbf{x}^T (I-W\otimes 1_n) \mathbf{x}
\\
& = \frac{\beta}{\alpha} \Big[ \alpha \sum_{k=1}^n f_k (x_k) + \frac{1}{2} \mathbf{x}^T (I-W\otimes 1_n) \mathbf{x}\Big] + \frac{(\alpha -\beta)}{2\alpha} \mathbf{x}^T (I-W \otimes 1_n) \mathbf{x}.
\end{split}
\end{equation*}
Since the last term is convex, and $G_{\alpha}$ is $\mu$-strongly convex, the above formula yields that $G_{\beta}$ is $\frac{\beta\mu}{\alpha}$-strongly convex. The proof is done.
\end{proof}

We now prove Theorem \ref{lem-1-2}.
\begin{proof}[Proof of Theorem \ref{lem-1-2}] For $0 < \alpha \leq \alpha_0$ we know that $G_{\alpha}$ is $\Big( \frac{\alpha}{\alpha_0}\Big)\mu$-strongly convex by Lemma \ref{lem-2-1}. Therefore
\begin{equation*}
G_{\alpha}(\mathbf{x}) \geq \frac{\alpha \mu}{2\alpha_0} \|\mathbf{x}-\mathbf{x}_*^{\alpha}\|^2
\end{equation*} 
for all $\mathbf{x} \in \mathbb{R}^{nm}$.
Taking $\mathbf{x}=0$ here, we get
\begin{equation*}
\frac{\alpha \mu}{2\alpha_0}\|\mathbf{x}_*^{\alpha}\|^2 \leq G_{\alpha}(0) = \alpha f (0),
\end{equation*}
which gives
\begin{equation*}
\|\mathbf{x}_*^{\alpha}\|^2 \leq \frac{2\alpha_0}{\mu}f (0).
\end{equation*}
This proves the first assertion. 

Next we are concerned with the smoothness property of $\mathbf{x}_*^{\alpha} \in \mathbb{R}^d$ with respect to the parameter $\alpha>0$. The function $G_{\beta}$ is $\frac{\mu\beta}{\alpha_0}$-strongly convex by Lemma \ref{lem-2-1}. Using this fact and that $ \mathbf{x}_*^{\beta}$ is a minimizer of $F_{\beta}$, we find  
\begin{equation*}
\begin{split}
&\beta \mathbf{F}(\mathbf{x}) + \frac{1}{2} \mathbf{x}^T (I-W) \mathbf{x} 
\\
&  \geq \frac{\mu \beta}{2\alpha_0}\|\mathbf{x}- \mathbf{x}_*^{\beta}\|^2 + \beta \mathbf{F}( \mathbf{x}_*^{\beta}) + \frac{1}{2} ( \mathbf{x}_*^{\beta})^T (I-W)  \mathbf{x}_*^{\beta}
\\
& = \frac{\mu \beta}{2\alpha_0}\|\mathbf{x}- \mathbf{x}_*^{\beta}\|^2 + (\beta-\alpha) \mathbf{F}( \mathbf{x}_*^{\beta})  +\Big[ \alpha \mathbf{F}( \mathbf{x}_*^{\beta}) + \frac{1}{2}( \mathbf{x}_*^{\beta})^T (I-W)  \mathbf{x}_*^{\beta}\Big].
\end{split}
\end{equation*}
Using the minimality of $\mathbf{x}_*^{\alpha}$ for $G_{\alpha}$ in the right hand side, we find
\begin{equation*}
\begin{split}
&\beta \mathbf{F}(\mathbf{x}) + \frac{1}{2} \mathbf{x}^T (I-W) \mathbf{x}   
\\
& \geq \frac{\mu \beta}{2\alpha_0}\|\mathbf{x}- \mathbf{x}_*^{\beta}\|^2 + (\beta- \alpha) \mathbf{F}( \mathbf{x}_*^{\beta}) +\Big[\alpha \mathbf{F}(\mathbf{x}_*^{\alpha}) + \frac{1}{2} ({\mathbf{x}_*^{\alpha}})^T (I-W) \mathbf{x}_*^{\alpha}\Big].
\end{split}
\end{equation*}
Taking $\mathbf{x}=\mathbf{x}_*^{\alpha}$, we find
\begin{equation*}  
\begin{split}
&\beta \mathbf{F}(\mathbf{x}_*^{\alpha}) + \frac{1}{2}(\mathbf{x}_*^{\alpha})^T (I-W) \mathbf{x}_*^{\alpha}
\\
& \geq \frac{\mu \beta}{2\alpha_0}\|\mathbf{x}_*^{\alpha}- \mathbf{x}_*^{\beta}\|^2 + (\beta - \alpha) \mathbf{F}( \mathbf{x}_*^{\beta}) + \alpha \mathbf{F}(\mathbf{x}_*^{\alpha}) + \frac{1}{2}(\mathbf{x}_*^{\alpha})^T (I-W) \mathbf{x}_*^{\alpha},
\end{split}
\end{equation*}
which gives
\begin{equation}\label{eq-2-1}
(\beta-\alpha) (\mathbf{F}(\mathbf{x}_*^{\alpha}) -\mathbf{F}( \mathbf{x}_*^{\beta})) \geq \frac{\mu \beta}{2\alpha_0}\|\mathbf{x}_*^{\alpha} - \mathbf{x}_*^{\beta}\|^2.
\end{equation}
Notice that
\begin{equation*}
\mathbf{F}(\mathbf{x}_*^{\alpha}) - \mathbf{F}( \mathbf{x}_*^{\beta}) =(\mathbf{x}_*^{\alpha}- \mathbf{x}_*^{\beta}) \cdot \int_0^1 \nabla \mathbf{F}( \mathbf{x}_*^{\beta} + s(\mathbf{x}_*^{\alpha}- \mathbf{x}_*^{\beta})) ds.
\end{equation*}
Therefore we have
\begin{equation*}
|\mathbf{F}(\mathbf{x}_*^{\alpha}) - \mathbf{F}( \mathbf{x}_*^{\beta}) | \leq C_1 \|\mathbf{x}_*^{\alpha} - \mathbf{x}_*^{\beta}\|.
\end{equation*}
This together with \eqref{eq-2-1} gives
\begin{equation*}
\frac{\mu \beta}{2\alpha_0}\|\mathbf{x}_*^{\alpha}- \mathbf{x}_*^{\beta}\|^2 \leq C_1 |\beta-\alpha| \|\mathbf{x}_*^{\alpha} - \mathbf{x}_*^{\beta}\|.
\end{equation*}
The proof is done.
\end{proof}

\section{Boundedness  property}\label{sec-4}
In this section, we make use of the properties of $\mathbf{x}_*^{\alpha}$ obtained in the previous section to study the sequence $\{\mathbf{x} (t)\}_{t \geq 0}$ of the decentralized gradient descent \eqref{eq-1-4}.
\begin{lem}\label{lem-3-1} Suppose that $G_{\alpha_0}$ is $\mu$-strongly convex for some $\alpha_0 >0$.
Assume that $\alpha (t) \leq\min\Big\{ \frac{1+\sigma_n (W)}{L},~\alpha_0\Big\}$. Then we have
\begin{equation*}
\|\mathbf{x}_{t+1} -\mathbf{x}_*^{\alpha (t+1)} \| \leq \|\mathbf{x}_t -\mathbf{x}_*^{\alpha (t)}\| + \frac{2\alpha_0 C_1}{\mu\alpha (t)} |\alpha (t+1) -\alpha (t)|.
\end{equation*}
\end{lem}
\begin{proof}
Note that $G_{\alpha}$ is $L_G$-smooth with $L_G = \alpha L + (1-\sigma_n (W))$. Thus, if $\alpha \leq \frac{1+\sigma_n (W)}{L}$, then we have $L_G \leq 2$. Note that 
\begin{equation*}
\begin{split}
&\|\mathbf{x}_{t+1}-\mathbf{x}_*^{\alpha}\|^2 
\\
&= \|\mathbf{x}_t- \nabla G_{\alpha}(\mathbf{x}_t) -\mathbf{x}_*^{\alpha}\|^2
\\
& = \|\mathbf{x}_t -\mathbf{x}_*^{\alpha}\|^2 - 2 \Big\langle \mathbf{x}_t -\mathbf{x}_*^{\alpha},~ \nabla G_{\alpha}(\mathbf{x}_t) -\nabla G_{\alpha}(\mathbf{x}_*^{\alpha}) \Big\rangle +  \|\nabla G_{\alpha}(\mathbf{x}_t) - \nabla G_{\alpha}(\mathbf{x}_*^{\alpha})\|^2.
\end{split}
\end{equation*}
Since $G_{\alpha}$ is convex and $L_G$-smooth, we have
\begin{equation*}
\langle \mathbf{x}_t -\mathbf{x}_*^{\alpha},~\nabla G_{\alpha}(\mathbf{x}_t) - \nabla G_{\alpha}(\mathbf{x}_*^{\alpha}) \rangle \geq \frac{1}{L_G} \|\nabla G_{\alpha}(\mathbf{x}_{t}) - \nabla G_{\alpha}(\mathbf{x}_*^{\alpha})\|^2.
\end{equation*}
Combining the above two estimates, we get
\begin{equation*}
\begin{split}
\|\mathbf{x}_{t+1} -\mathbf{x}_*^{\alpha}\|^2 & \leq \|\mathbf{x}_t -\mathbf{x}_*^{\alpha}\|^2 - \Big( \frac{2}{L_G} -1 \Big) \|\nabla G_{\alpha}(\mathbf{x}_t) - \nabla G_{\alpha}(\mathbf{x}_*^{\alpha})\|^2
\\
&\leq \|\mathbf{x}_t -\mathbf{x}_*^{\alpha}\|^2.
\end{split}
\end{equation*}
Using this with the triangle inequality and Theorem \ref{lem-1-2}, we deduce
\begin{equation*}
\begin{split}
\|\mathbf{x}_{t+1} -\mathbf{x}_*^{\alpha (t+1)}\| & \leq \|\mathbf{x}_{t+1} -\mathbf{x}_*^{\alpha (t)}\| + \|\mathbf{x}_*^{\alpha (t)} - \mathbf{x}_*^{\alpha (t+1)}\|
\\
& \leq \|\mathbf{x}_{t+1} -\mathbf{x}_*^{\alpha (t)}\| +  \frac{2\alpha_0 C_1}{\mu\alpha (t)}|\alpha (t) -\alpha (t+1)|.
\end{split}
\end{equation*}
The proof is done.

\end{proof}

In order to prove Theorem \ref{thm-3-3}, we recall the following result from \cite{CK}.
\begin{thm}[\cite{CK}]\label{thm-3-2} 
Suppose that  the total cost function $f$ is $\mu$-strongly convex for some $\mu >0$ and each $f_i$ is $L$-smooth for $1 \leq i \leq m$. If $\{\alpha(t)\}_{t\in\mathbb{N}_0}$ is non-increasing stepsize satisfying $\alpha (0) < \frac{\eta (1-\beta)}{L ( \eta + L )}$, then we have
\begin{equation}\label{eq-2-30}
\|\bar{\mathbf{x}}(t)-\mathbf{x}_* \|  \leq R \quad and \quad \|\mathbf{x}(t)-\bar{\mathbf{x}}(t) \| \leq \frac{\eta R}{L}<R, \quad \forall t \geq 0.
\end{equation}
Here $\eta = \frac{\mu L}{\mu +L}$ and 
 a finite value $R>0$ is defined as
\begin{equation*}
R = \max\left\{ \|\bar{\mathbf{x}}(0) - \mathbf{x}_*\|,~\frac{L}{\eta}\|\mathbf{x}(0)- \bar{\mathbf{x}}(0)\|,~\frac{\sqrt{m} D\alpha(0)}{\eta (1-\beta)/L  - (\eta +L) \alpha (0)}\right\}.
\end{equation*} 
\end{thm}

 Using the above lemma, we give the proof of the main theorem on the boundedness  property of the sequence $\{\mathbf{x}(t)\}_{t \geq 0}$.

\begin{proof}[Proof of Theorem \ref{thm-3-3}] Scine $F \in \mathbf{G}_{\alpha}$, the function $G_{\alpha}$ is $\mu$-strongly convex for some $\mu >0$. Then we know that $f$ is $\frac{\mu}{\alpha_0}$-strongly convex by Theorem \ref{thm-2-2}. We set the following constants
\begin{equation*}
\eta = \frac{(\mu/\alpha_0) L}{(\mu/\alpha_0) +L}\quad \textrm{and}\quad \mathbf{a}_c = \frac{\frac{\mu}{\alpha_0} (1-\beta)}{ 2L\Big( \frac{\mu}{\alpha_0} + L\Big)} = \frac{\mu (1-\beta)}{4L (\mu + \alpha_0 L)}.
\end{equation*}
Take a number $t_0 \in \mathbb{N}$ such that $\alpha (t_0) \geq \mathbf{a}_c$ and $\alpha (t_0 +1) < \mathbf{a}_c$. 

For $t \leq t_0$ we apply Lemma \ref{lem-3-1} to find
\begin{equation*}
\begin{split}
\|\mathbf{x}_{t+1} -\mathbf{x}_*^{\alpha (t+1)}\| &\leq \|\mathbf{x}_t -\mathbf{x}_*^{\alpha (t)}\| +\frac{2C_1}{\alpha (t)}|\alpha (t+1) -\alpha (t)|
\\
&\leq \|\mathbf{x}_t -\mathbf{x}_*^{\alpha (t)}\| +\frac{2C_1}{\mathbf{a}_c}|\alpha (t+1) -\alpha (t)|.
\end{split}
\end{equation*}
Thus, 
\begin{equation*}
\begin{split}
\|\mathbf{x}_{t_0+1} -\mathbf{x}_*^{\alpha (t)}\| & \leq \|\mathbf{x}_0 -\mathbf{x}_{*}^{\alpha (1)}\| +\frac{2C_1}{\mathbf{a}_c} \sum_{s=0}^{t_0} |\alpha (s+1) -\alpha (s)|
\\
& = \|\mathbf{x}_0 -\mathbf{x}_*^{\alpha (0)}\| + \frac{2C_1}{\mathbf{a}_c} (\alpha (0) - \alpha (t_0+1)).
\end{split}
\end{equation*}
Note that
\begin{equation*}
\|\mathbf{x}_{t_0 +1}\| \leq \|\mathbf{x}_0 -\mathbf{x}_*^{\alpha (0)}\| + \frac{2C_1}{\mathbf{a}_c}\alpha (0) + \|\mathbf{x}_*^{\alpha (t_0 +1)}\|.
\end{equation*}
Therefore 
\begin{equation*}
\|\bar{\mathbf{x}}_{t_0 +1} - \mathbf{x}_{t_0+1}\| \leq \|\mathbf{x}_0 -\mathbf{x}_*^{\alpha (0)}\| + \frac{2C_1}{\mathbf{a}_c}\alpha (0) + \|\mathbf{x}_*^{\alpha (t_0 +1)}\|
\end{equation*}
and
\begin{equation*}
\begin{split}
\|\bar{\mathbf{x}}_{t_0 +1} -\mathbf{x}_*\| & \leq \|\bar{\mathbf{x}}_{t_0 +1}\| + \|\mathbf{x}_*\|
\\
&\leq \|\mathbf{x}_0 -\mathbf{x}_*^{\alpha (0)}\| +\frac{2C_1}{\mathbf{a}_c}\alpha (0) + \|\mathbf{x}_*^{\alpha (t_0 +1)}\| +\|\mathbf{x}_*\|.
\end{split}
\end{equation*}
Now we use the result of Theorem \ref{thm-3-2} to deduce that
\begin{equation*}
R = \max\left\{ \|\bar{\mathbf{x}}_{t_0 +1} - \mathbf{x}_*\|,~\frac{L}{\eta}\|\mathbf{x}_{t_0+1}- \bar{\mathbf{x}}_{t_0+1}\|,~\frac{\sqrt{m} D\alpha(t_0+1)}{\eta (1-\beta)/L  - (\eta +L) \alpha (t_0 +1)}\right\},
\end{equation*} 
then we have
\begin{equation}\label{eq-2-30}
\|\bar{\mathbf{x}}(t)-\mathbf{x}_* \|  \leq R \quad and \quad \|\mathbf{x}(t)-\bar{\mathbf{x}}(t) \| \leq \frac{\eta R}{L}<R, \quad \forall t \geq t_0 +1.
\end{equation}
Here $\eta = \frac{\mu L}{\mu +L}$. We also check that
\begin{equation*}
\begin{split}
\frac{\sqrt{m} D\alpha(t_0+1)}{\eta (1-\beta)/L  - (\eta +L) \alpha (t_0 +1)}& = \frac{\sqrt{m}D}{\frac{\eta (1-\beta)}{L \alpha (t_0+1)} - (\eta +L)}
\\
& \leq \frac{\sqrt{m}D}{\frac{\eta (1-\beta)}{L\mathbf{a}_c} - (\eta+L)} = \frac{\sqrt{m}D}{2 (\eta +L)}.
\end{split}
\end{equation*}
The proof is finished.
\end{proof}

\section{Numerical experiments}\label{sec-5}

In this section, we provide numerical experiments supporting the theoretical results of this paper. 

First we set $m=3$ and $n=2$ for problem \eqref{prob}. For each $1 \leq k \leq m$, we choose the local cost $f_k$ as
\begin{equation*}
f_k (x) = \frac{1}{2} x^T A_k x + B_k^T x.
\end{equation*}
Here the matrix $A_k$ is an $n\times n$ symmetrix matrix chosen as
\begin{equation*}
A_k = \ep I_{n\times n} + (R_k +R_k^T),
\end{equation*}
where $\ep >0$ and each element of $R_k$ is chosen randomly from $[-1,1]$ with uniform distribution. Also, each element of  $B_k \in \mathbb{R}^n$ is randomly chosen in $[-1,1]$ with uniform distribution. We choose $W$ as
\begin{equation*}
W = \begin{pmatrix} 1/2 & 1/4& 1/4 \\ 1/4& 1/2 & 1/4 \\ 1/4& 1/4 &1/2\end{pmatrix}.
\end{equation*}
In order to verify the result of Theorem \ref{thm-3-3}, we compute the following constants:
\begin{equation*}
\begin{split}
\alpha_A& = \textrm{a possibly large value $\alpha_0>0$ such that $F \in \mathbf{G}_{\alpha_0}$}
\\
\alpha_L & = \frac{1+\lambda_m (W)}{L}.
\end{split}
\end{equation*}
The detail for computing the above values are explained below:
\begin{itemize}
\item To compute $\alpha_A$ we choose a large $N \in \mathbb{N}$ and set
\begin{equation*}
\alpha_A = \sup_{k \in \mathbb{N}} \Big\{ \frac{k}{N}>0~|~ F \in \mathbf{G}_{(k/N)}\Big\}.
\end{equation*}
Since $G_{(k/N)} (x)$ is a quadratic function, we may check the positivity of all eigenvalues of $\nabla^2 G_{(k/N)}(x)$ to determine if the function $G_{(k/N)}$ is strongly convex. 
\item To find the smallest value of $L>0$, we compute $L_k = \|A_k\|_{\infty}$ by using the eigenvalues of $A_k$. Then we set $L = \sup_{1 \leq k \leq m} L_k$.
\end{itemize}
In our experiment, the constants are computed as
\begin{equation*}
\alpha_A \simeq 0.0799\quad \textrm{and}\quad 
\alpha_L  \simeq 0.1721
\end{equation*}
with $L\simeq 7.2615 $ and $\lambda_m (W) = 0.25$.

Since $\alpha_A < \alpha_L$, the range \eqref{eq-1-61} of the stepsize $\alpha(t)$ guaranteed by Theorem \ref{thm-3-3} is given as
\begin{equation*}
\alpha (0) \leq \alpha_A.
\end{equation*} 
We take the constant stepsize $\alpha (t) \equiv \alpha >0$ with various choices of $\alpha$ given as
\begin{equation*}
\{0.5 \alpha_A, ~0.95 \alpha_A,~ 0.99 \alpha_A,~1.01 \alpha_A, ~1.01 \}. 
\end{equation*}
For each time step $t \geq 0$, we measure the following error
\begin{equation*}
R(t) = \sum_{k=1}^m \|x_k (t) -x_*\|,
\end{equation*}
where $x_k (t)$ is the state of $k$ in \eqref{eq-1-1} and $x_*$ is the optimizer of \eqref{prob}.
\begin{figure}[!htbp]
\centering
\includegraphics[width=7.4cm]{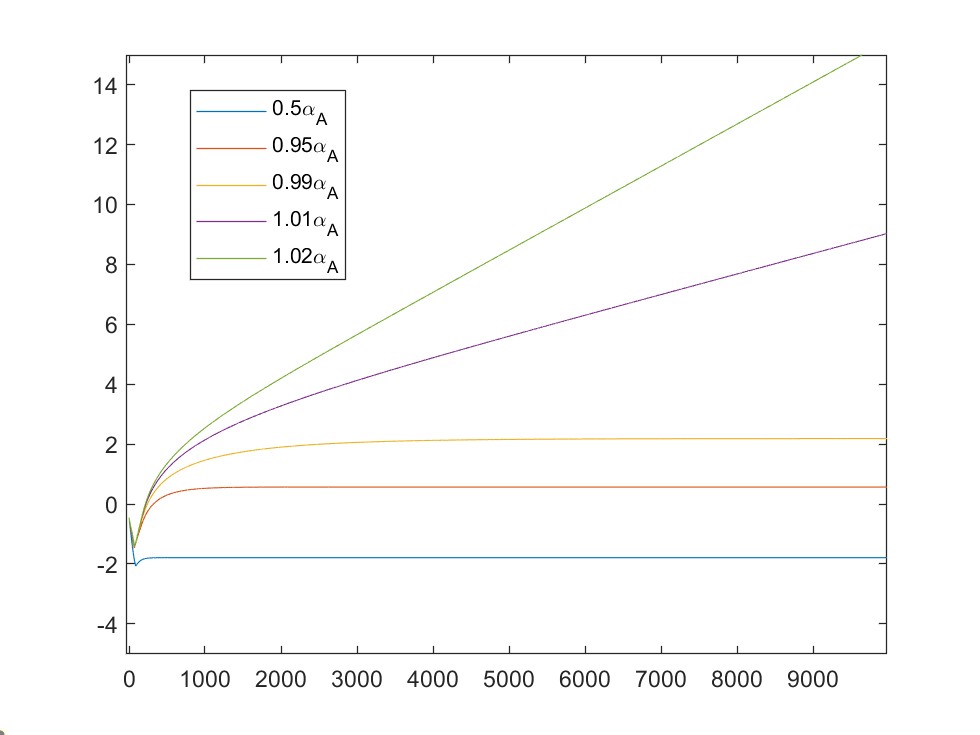}
\includegraphics[width=7.4cm]{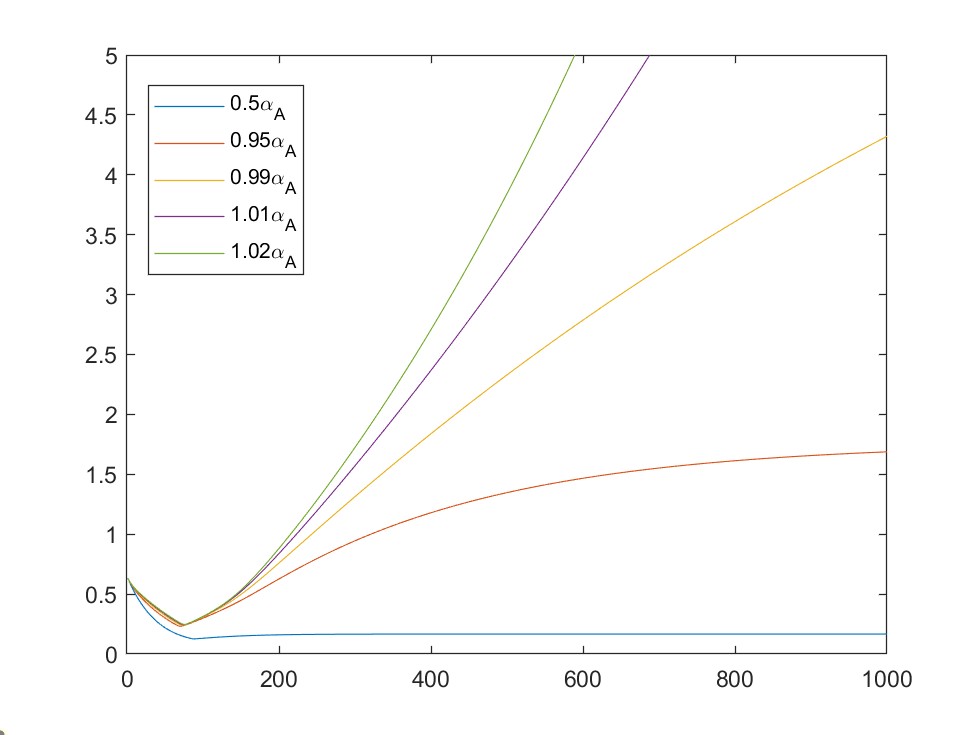}
\caption{The left (resp., right) figure corresponds to the graph of value $\log R(t)$ (resp.,  $R(t)$) with respect to $t \geq 0$.}
\label{figure1}
\end{figure}
The result shows that the states $\{x_k (t)\}_{k=1}^m$ are uniformly bounded for the three cases $\alpha \in \{0.5 \alpha_A, 0.95 \alpha_A, 0.99 \alpha_A\}$ as expected by Theorem \ref{thm-3-3}. Meanwhile, the states $\{x_k (t)\}$ diverges as $t \rightarrow \infty$ for the choices $\alpha \in \{1.01 \alpha_A, 1.02\alpha_A\}$ which are larger than the value $\alpha_A$. This shows the sharpness of the result of Theorem \ref{thm-3-3}.

\end{document}